\documentclass[11pt,a4paper]{article}
\usepackage{amsfonts,amsgen,amsmath,amstext,amsbsy,amsopn,amsthm,amsfonts,amssymb,amscd}
\usepackage{epsf,epsfig}
\usepackage{float}
\usepackage{ebezier,eepic}
\usepackage{color}
\usepackage{tikz}
\usepackage{multirow}
\usepackage{graphicx}
\usepackage{subfigure}
\newtheorem{thm}{Theorem}[section]

\newtheorem{prop}[thm]{Proposition}
\newtheorem{lem}[thm]{Lemma}
\newtheorem{false statement}{False statement}

\theoremstyle{definition}
\newtheorem{definition}[thm]{Definition}

\newtheorem{conj}[thm]{Conjecture}

\makeatletter \@addtoreset{equation}{section}

\def\hm{\mathcal{M}}

\def\hh{\mathcal{H}}
\def\hk{\mathcal{K}}

\begin{document}
\title{\bf\Large The Maximum Number of Cliques in Hypergraphs without Large Matchings}
\date{}
\author{ Erica L.L. Liu$^1$, Jian Wang$^2$\\[10pt]
$^{1}$Center for Applied Mathematics\\
Tianjin University\\
Tianjin 300072, P. R. China\\[6pt]
$^{2}$Department of Mathematics\\
Taiyuan University of Technology\\
Taiyuan 030024, P. R. China\\[6pt]
E-mail:  $^1$liulingling@tju.edu.cn, $^2$wangjian01@tyut.edu.cn
}

\maketitle

\begin{abstract}
Let $[n]$ denote the set $\{1, 2, \ldots, n\}$ and $\mathcal{F}^{(r)}_{n,k,a}$ be an $r$-uniform hypergraph on the vertex set $[n]$ with edge set consisting of all the $r$-element subsets of $[n]$ that contains at least $a$ vertices in $[ak+a-1]$.
For $n\geq 2rk$, Frankl proved that $\mathcal{F}^{(r)}_{n,k,1}$ maximizes the number of  edges   in $r$-uniform hypergraphs on $n$ vertices with the matching number at most $k$. Huang, Loh and Sudakov considered a multicolored version of the Erd\H{o}s matching conjecture, and provided a sufficient condition on the number of edges for a multicolored hypergraph to contain a rainbow matching of size $k$.
In this paper, we show that $\mathcal{F}^{(r)}_{n,k,a}$ maximizes the number of  $s$-cliques in $r$-uniform hypergraphs on $n$ vertices with the matching number at most $k$ for sufficiently large $n$, where $a=\lfloor \frac{s-r}{k} \rfloor+1$. We also obtain a condition on the number of $s$-clques for a multicolored $r$-uniform hypergraph to contain a rainbow matching of size $k$, which reduces to the condition of Huang, Loh and Sudakov when $s=r$.
\end{abstract}


\medskip

\section{Introduction}
 An {\it $r$-graph} (or an {\it $r$-uniform hypergraph}) is a pair $\hh=(V, E)$, where $V=V(\hh)$ is a finite set of vertices, and $E=E(\hh)\subset{V\choose r}$ is a family of $r$-element subsets of $V$. We often identify $E(\hh)$ with $\hh$. For any $S\subset V(\hh)$, let $\hh[S]$ be the subhypergraph of $\hh$ induced by $S$ and let  $\hh-S$ denote the subhypergraph of $\hh$ induced by $V(\hh)\setminus S$. For any $S\subset V(\hh)$ with $|S|<r$, let
$$N_{\hh}(S) =\left\{T\in \binom{V(\hh)}{r-|S|}\colon S\cup T\in \hh \right\}$$
and $\deg_{\hh}(S)=|N_{\hh}(S)|$.
We call the elements in $N_{\hh}(S)$ the neighbors of $S$ in $\hh$ and call  $\deg_{\hh}(S)$ the degree of $S$ in $\hh$. For $S=\{v\}$, we often use $H-v$, $N_{\hh}(v)$ and $\deg_{\hh}(v)$ instead of $\hh-\{v\}$, $N_{\hh}(\{v\})$ and $\deg_{\hh}(\{v\})$, respectively.
For any $s\geq r$, an $s$-clique of $\hh$ is a subhypergraph of $\hh$ on $s$ vertices in which every subset of $r$ vertices is an edge of $\hh$. Let $\hk^r_{s}(\hh)$ denote the family of all the $s$-cliques of $\hh$ and let $K^r_{s}(\hh)$ be the cardinality of $\hk^r_{s}(\hh)$. For any $u\in V(\hh)$, we   use $K^r_{s}(u,\hh)$ to denote the number of $s$-cliques in $\hh$ containing $u$. A {\it matching} in $\hh$ is a collection of pairwise disjoint edges of $\hh$. The matching number of $\hh$, denoted by $\nu(\hh)$, is the size of a maximum matching in $\hh$.


\begin{definition}
Let $n, k, r, a$ be positive integers with $n\geq r\geq a$. Define
$$\mathcal{F}^{(r)}_{n,k,a}=\left\{F\in{[n]\choose r}: \left|F\cap[ak+a-1]\right|\geq a\right\}.$$
\end{definition}

Clearly, we have $\nu(\mathcal{F}^{(r)}_{n,k,a})\leq k$. Otherwise, we may assume that $\{E_1,E_2,\ldots,E_{k+1}\}$ is a matching of size $k+1$ in $\mathcal{F}^{(r)}_{n,k,a}$, then we have
 \[
 \left|[ak+a-1]\right|\geq \sum_{i=1}^{k+1} |[ak+a-1] \cap E_i|\geq (k+1)a,
 \]
 a contradiction.

 In 1965, Erd\H{o}s \cite{erdos4} proposed the following  conjecture.

\begin{conj}[The Erd\H{o}s matching conjecture \cite{erdos4}]\label{matchingconj}
Let $\hh$ be an $r$-graph on $n$ vertices with $\nu(\hh)\leq k$. Then
\[
|\hh|\leq \max\left\{ |\mathcal{F}^{(r)}_{n,k,1}|,|\mathcal{F}^{(r)}_{n,k,r}|\right\}.
\]
\end{conj}

In 2013, Frankl proved that  Conjecture \ref{matchingconj} holds for $n\geq (2k+1)r-k$.

\begin{thm}[Frankl \cite{frankl1}]\label{matchinglem}
Let $\hh$ be an $r$-graph on $n$ vertices with $\nu(\hh)\leq k$. If $n\geq (2k+1)r-k$, then $|\hh|\leq  |\mathcal{F}^{(r)}_{n,k,1}|$.
\end{thm}

For recent results on Conjecture \ref{matchingconj}, we refer the reader to \cite{frankl1,frankl2,frankl3}. For ordinary
 graphs, Alon and Shikhelman \cite{alon} introduced a generalization of the usual Tur\'{a}n problem, which is often called the generalized Tur\'{a}n problem. Given two graphs $T$ and $H$, the {\it generalized Tur\'{a}n number}, denoted by $ex(n,T,H)$, is defined to be the maximum number of copies of $T$ in an $H$-free graph on $n$ vertices. The first result in this
 direction is due to Zykov \cite{zykov} and independently to Erd\H{o}s \cite{erdos}, who determined $ex(n,K_s,K_t)$.  The second author \cite{wj} determined $ex(n,K_s,M_{k+1})$, where $M_{k+1}$ is a matching of size $k+1$. Recently, the study of the generalized Tur\'{a}n problem has received much attention, see \cite{alon,gerbner2,gerbner3,gishboliner,gowers,jiema,luo,boning}.

Motivated by the Erd\H{o}s matching conjecture and the generalized Tur\'{a}n problem, we determine the maximum number of $s$-cliques in an $r$-graph on $n$ vertices with matching number at most $k$.

\begin{thm}\label{th1}
Let $n, k, r, s$ be integers and $\hh$ be an $r$-graph on $n$ vertices with $\nu(\hh)\leq k$.

\begin{itemize}
  \item[(I)] If $r\leq s\leq k+r-1$ and $n\geq 4(er)^{s-r+2}k$, then $K_s^r(\hh)\leq K_s^r(\mathcal{F}^{(r)}_{n,k,1})$;
  \item[(II)] If $k+r\leq s\leq (r-1)(k+1)$ and $n\geq 4r^2k(er/(a-1))^{s-r+a}$, then
$K_s^r(\hh)\leq K_s^r(\mathcal{F}^{(r)}_{n,k,a})$,
where $a=\lfloor \frac{s-r}{k} \rfloor+1$;
  \item[(III)] If $(r-1)k+r\leq s\leq rk+r-1$ and $n\geq rk+r-1$, then
$K_s^r(\hh)\leq K_s^r(\mathcal{F}^{(r)}_{n,k,r})$. Moreover, if $K_s^r(\hh)< K_s^r(\mathcal{F}^{(r)}_{n,k,r})$, then $K_s^r(\hh)\leq {rk+r-1\choose s}-{rk-1\choose s-r}$.
\end{itemize}
\end{thm}

Based on the construction $\mathcal{F}^{(r)}_{n,k,a}$, one sees that the upper bounds in  Theorem \ref{th1} are tight. Let $k, r, s$ be integers with $r\leq s\leq (r-1)(k+1)$,  $a=\lfloor \frac{s-r}{k} \rfloor+1$ and
\[
n^*(k,r,s)=\left(\frac{r}{a}\right)^{\frac{s-r+a}{r-a}}\left(\frac{rk+r-1-s}{s}\right).
\]
If $n\leq n^*(k,r,s)$, we have
\begin{align*}
K_s^r(\mathcal{F}^{(r)}_{n,k,a})&\leq\binom{ak+a-1}{s-r+a}\binom{n-s+r-a}{r-a} \\[5pt]
&\leq \left(\frac{a}{r}\right)^{s-r+a}\binom{rk+r-1}{s-r+a}\frac{n^{r-a}}{(r-a)!} \\[5pt]
&< \binom{rk+r-1}{s-r+a}\left(\frac{rk+r-s-1}{s}\right)^{r-a}\\[5pt]
&\leq \binom{rk+r-1}{s} =K_s^r(\mathcal{F}^{(r)}_{n,k,r}).
\end{align*}
Note that $\mathcal{F}^{(r)}_{n,k,r}$ is an $r$-graph on $n$ vertices with matching number at most $k$.  Since $K_s^r(\mathcal{F}^{(r)}_{n,k,r}) > K_s^r(\mathcal{F}^{(r)}_{n,k,a})$ for  $n\leq n^*(k,r,s)$, (I) and (II) in Theorem \ref{th1} hold if and only if $n\geq n_0(k,r,s)$ for some integer $n_0(k,r,s)>n^*(k,r,s)$.


Huang, Loh and Sudakov \cite{huang} considered a multi-colored generalization of the Erd\H{o}s matching conjecture and they proved the following theorem.

\begin{thm}[Huang, Loh and Sudakov \cite{huang}]\label{huang-loh-sudakov}
Let $\mathcal{F}_1,\ldots, \mathcal{F}_k$ be $r$-graphs on the vertex set $[n]$, where $k\leq \frac{n}{3r^2}$, and for any $i$, $|\mathcal{F}_i|>|\mathcal{F}^{(r)}_{n,k-1,1}|$. Then there exist pairwise disjoint edges $F_1\in\mathcal{F}_1, \ldots, F_k\in\mathcal{F}_k$.
\end{thm}

In this paper, we generalize their result by loosing the conditions on $\mathcal{F}_i$.

\begin{thm}\label{th3}
Let $n, k, r, t$ be integers  such that $r\leq t\leq k+r-2$ and $n\geq 4k(t-r+2)(er)^{t-r+2}$. Let $\mathcal{F}_1, \mathcal{F}_2, \ldots, \mathcal{F}_k$ be $r$-graphs on the vertex set $V$ of size $n$. If for any $i\in \{1,2,\ldots,k\}$, there exists some $s\in \{r,r+1,\ldots,t\}$ such that $K_s^r(\mathcal{F}_i)>K_s^r(\mathcal{F}^{(r)}_{n,k-1,1})$.
Then there exist pairwise disjoint edges $F_1\in\mathcal{F}_1, \ldots, F_k\in\mathcal{F}_k$.
\end{thm}

To prove the above theorem, we need some estimates on the binomial coefficients, which are listed below. Let $a, b$ and $c$ be integers satisfying $a\geq b\geq c\geq 0$. Then the following inequalities hold:
\begin{align}
{a\choose b}&\leq \left(\frac{ea}{b}\right)^b,\label{eq1}\\[5pt]
{b\choose c}&\leq \left(\frac{b}{a}\right)^c{a\choose c},\label{eq2}\\[5pt]
{a\choose c}&\leq\left(\frac{a-c}{b-c}\right)^c{b\choose c},\label{eq3}\\[5pt]
{a\choose c}&\leq\left(\frac{ea}{b}\right)^c{b\choose c}.\label{eq4}
\end{align}
When $b$ is close to $c$, the inequality \eqref{eq4} gives a better upper bound on $\binom{a}{c}$ than the inequality \eqref{eq3}. Let $p$ be a positive integer and $x\in (0,\frac{1}{p}]$. Then we have
\begin{align}\label{eq5}
(1+x)^p\leq 1+p^2x.
\end{align}

By the definition of $\mathcal{F}^{(r)}_{n,k,1}$ we have
\begin{align*}
K_s^r(\mathcal{F}^{(r)}_{n,k,1})=\sum_{j=s-r+1}^s \binom{k}{j}\binom{n-k}{s-j}.
\end{align*}
It is easy to check that
\begin{align}\label{eq6}
K_s^r(\mathcal{F}^{(r)}_{n-1,k-1,1})+K_{s-1}^r(\mathcal{F}^{(r)}_{n-1,k-1,1})=K_s^r(\mathcal{F}^{(r)}_{n,k,1}).
\end{align}

In Section 2, we  prove
(I) of Theorem \ref{th1}. The proofs of
(II) and (III) of Theorem \ref{th1} will be given in Section 3. Theorem \ref{th3} will be
proved in Section 4.

\section{The maximum number of $s$-cliques with $s\leq k+r-1$}

In this section, we determine the maximum number of $s$-cliques in an $r$-graph $\hh$ with $\nu(\hh)\leq k$ when $s\leq k+r-1$. We need the following result
 due to Huang, Loh and Sudakov \cite{huang}.

\begin{lem}[Huang, Loh and Sudakov \cite{huang}]\label{le1}
Let $n,k,r$ be integers such that $rk\leq n$ and $\hh$ be an $r$-graph on $n$ vertices. If $\hh$ has $k$ distinct vertices $v_1, v_2, \ldots, v_k$ with   $\deg_\mathcal{H}(v_i)>2(k-1){n-2\choose r-2}$, then $\hh$ contains a matching of size $k$.
\end{lem}

Theorem \ref{th1} (I) will be proved by   induction. The following lemma is
the basis of the induction.

\begin{lem}\label{lea1}
Let $r,s$ be positive integers such that $r\leq s$ and $\hh$ be an $r$-graph on $n$ vertices with $\nu(\hh)\leq s-r+1$. For $n\geq 4(s-r+1)(er)^{s-r+2}$, we have $K_s^r(\hh)\leq K_s^r(\mathcal{F}^{(r)}_{n,s-r+1,1})$.
\end{lem}

\begin{proof}
 Let $\hm=\{E_1,E_2,\ldots,E_p\}$ be a maximum matching in $\hh$ and $S$ be the set of vertices that are covered by $\hm$. Since $\nu(\hh)\leq s-r+1$, we have $p\leq s-r+1$.
Let
\begin{align*}
&X=\left\{x\in S\colon \deg_{\hh}(x) > 2(s-r+1){n-2\choose r-2}\right\},\\[5pt]
&Y=\left\{x\in S\colon \deg_{\hh}(x) > r(s-r+1){n-2\choose r-2}\right\}.
\end{align*}
Clearly, $Y\subset X$. By Lemma \ref{le1}, we have $|X|\leq s-r+1$. Thus, $|Y|\leq s-r+1$. Now the proof splits into two cases depending on the size of $Y$.

Case 1. $|Y|=s-r+1$. We claim that every edge of $\hh$ contains at least one vertex in $Y$. Otherwise, assume that $E$ is an edge of $\hh$ that is disjoint from $Y$.
Let $Y=\{x_1, x_2, \ldots, x_{s-r+1}\}$.  {For each $i=1,2,\ldots,s-r+1$, since there are at most $r{n-2\choose r-2}$ sets in $N_{\hh}(x_i)$ that intersects $E$, it follows that}
 \begin{align*}
 \deg_{\hh- E}(x_i) &> r(s-r+1){n-2\choose r-2}-r{n-2\choose r-2}\\[5pt]
 &\geq 2(s-r){n-r-2\choose r-2}.
 \end{align*}
 By Lemma \ref{le1}, $\hh- E$ contains a matching of size $s-r+1$. Then, $\hh$ contains a matching of size $s-r+2$, {which contradicts the fact that $\nu(\hh)\leq s-r+1$.} Thus, the claim holds.
Therefore, $\hh$ is isomorphic to a subhypergraph of $\mathcal{F}^{(r)}_{n,s-r+1,1}$, and we have $K_s^r(\hh)\leq K_s^r(\mathcal{F}^{(r)}_{n,s-r+1,1})$.

Case 2. $|Y|\leq s-r$.  Clearly, each $s$-clique in $\hh$ contains at least $s-r+1$ vertices in $S$. Otherwise, we obtain a matching of size $p+1$ in $\hh$, which contradicts the fact that $\hm$ is a maximum matching in $\hh$. Now we count the number of $s$-cliques in $\hh$ as follows.
First, we choose a set $A$ of $(s-r+1)$ vertices in $S$ and there are at most ${|S|\choose s-r+1}$ {choices} for $A$. Then choose an $(r-1)$-element subset $B$ of $V(\hh)$, which may form an $s$-clique in $\hh$ together with $A$. Consequently, $B$ has to be a common neighbor of all the vertices in $A$.
Since $|A|>|Y|$, there exists some $x\in A$ that falls in $S\setminus Y$. If $A\subset X$, the number of choices for $B$ is at most $(s-r+1)r{n-2\choose r-2}$. If $A$ is not contained in $X$, the number of choices for $B$ is at most $2(s-r+1){n-2\choose r-2}$.
 Thus we have
\[
K_s^r(\hh)\leq (s-r+1)r{n-2\choose r-2}{|X|\choose s-r+1}+2(s-r+1){n-2\choose r-2}{|S|\choose s-r+1}.
\]
Since $|X|\leq s-r+1$ and $|S|=pr\leq r(s-r+1)$, we find that
\begin{align}\label{eqs5}
K_s^r\leq (s-r+1)r{n-2\choose r-2}+2(s-r+1){n-2\choose r-2}{(s-r+1)r\choose s-r+1}.
\end{align}
Using the inequality \eqref{eq1}, we get
\begin{align}\label{eqs6}
{(s-r+1)r\choose s-r+1} \leq (er)^{s-r+1}.
\end{align}
Combining \eqref{eqs5} and \eqref{eqs6}, we obtain that
\begin{align}\label{eqs7}
K_s^r(\hh)\leq3(s-r+1)(er)^{s-r+1}{n-2\choose r-2}.
\end{align}
By the inequality \eqref{eq3}, we have
\[
{n-2\choose r-2}\leq\left(\frac{n-r}{n-s}\right)^{r-2}{n-s+r-2\choose r-2}.
\]
Therefore, \eqref{eqs7} implies that
\begin{align}\label{eqs8}
K_s^r(\hh)\leq3(s-r+1)(er)^{s-r+1}\left(\frac{n-r}{n-s}\right)^{r-2}{n-s+r-2\choose r-2}.
\end{align}
Applying \eqref{eq5} gives
\begin{align}\label{eqs9}
\left(\frac{n-r}{n-s}\right)^{r-2} =\left(1+\frac{s-r}{n-s}\right)^{r-2} \leq 1+\frac{(r-2)^2(s-r)}{n-s}.
\end{align}
It follows from \eqref{eqs8} and \eqref{eqs9} that
\begin{align}\label{eqs10}
K_s^r(\hh)\leq3(s-r+1)(er)^{s-r+1}\left(1+\frac{(r-2)^2(s-r)}{n-s}\right){n-s+r-2\choose r-2}.
\end{align}
Since $n\geq 4(s-r+1)(er)^{s-r+2}$, it is easily checked that
\begin{align}\label{eqs11}
\frac{(r-2)^2(s-r)}{n-s} \leq \frac{1}{3}
\end{align}
and
\begin{align}\label{eqs12}
4(s-r+1)(er)^{s-r+1}\leq\frac{n-s+r-1}{r-1}.
\end{align}
Combining \eqref{eqs10}, \eqref{eqs11} and \eqref{eqs12}, we deduce that
\begin{align*}
K_s^r(\hh)\leq\frac{n-s+r-1}{r-1}{n-s+r-2\choose r-2}={n-s+r-1\choose r-1}=K_s^r(\mathcal{F}^{(r)}_{n,s-r+1,1}).
\end{align*}
This completes the proof.
\end{proof}

Now we are ready to prove Theorem \ref{th1} (I).

\begin{proof}[Proof of Theorem \ref{th1} (I)]
Let $n, r$ be positive integers. We shall proceed by double induction on $s$ and $k$. Recall the condition $r\leq s\leq k+r-1$. For $s=r$ and all $k\geq 1$, the result follows from Theorem \ref{matchinglem}. For all $s\geq r$ and $k=s-r+1$, the result follows from Lemma \ref{lea1}.  Now we assume that the assertion holds for $(s-1,k-1)$ and $(s,k-1)$.
Let $\hh$ be an $r$-graph on $n$ vertices with $n\geq 4k(er)^{s-r+2}$. Without loss of generality, we assume that $\nu(\hh)=k$. Let $\hm=\{E_1,E_2,\ldots,E_k\}$ be a maximum matching in $\hh$ and $S$ be the set of vertices covered by $\hm$.

If there exists a vertex $u\in V(\hh)$ such that $\nu(\hh-u)=k-1$, then by the induction hypothesis we have $K_s^r(\hh-u)\leq K_s^r(\mathcal{F}^{(r)}_{n-1, k-1,1})$. Again, by the induction hypothesis, we have
$$K_s^r(u,\hh) \leq K_{s-1}^r(\hh-u)\leq K_{s-1}^r(\mathcal{F}^{(r)}_{n-1, k-1,1}).$$
From the equality (\ref{eq6}), it follows that
\begin{align*}
K_s^r(\hh)&=K_s^r(\hh-u)+K_s^r(u,\hh)\\[5pt]
&\leq K_s^r(\mathcal{F}^{(r)}_{n-1, k-1,1})+K_{s-1}^r(\mathcal{F}^{(r)}_{n-1, k-1,1})\\[5pt]
&=K_{s}^r(\mathcal{F}^{(r)}_{n, k,1}).
\end{align*}
Thus, we have shown that Theorem \ref{th1} (I) holds if there exists a vertex $u\in V(\hh)$ such that $\nu(\hh-u)=k-1$.

Now we consider the case that $\nu(\hh-u)=k$ for any $u\in V(\hh)$. We claim that the maximum degree in $\hh$ is at most $rk{n-2\choose r-2}$.  Let $u\in V(\hh)$ and $\hm'$ be a matching of size $k$ in $\hh-u$. For each edge $F$ in $\hh$ with $u\in F$,  it is easy to see that $|F\cap(\cup_{E\in \hm'}E)| \geq 1$. It follows that the maximum degree of $\hh$ is at most $rk{n-2\choose r-2}$.

Let $Y$ be the set of vertices in $S$ with degree greater than $2k{n-2\choose r-2}$. If $|Y|\geq k+1$,  by Lemma \ref{le1} we obtain a matching of size $k+1$ in $\hh$, contradicting the assumption that $k$ is the size of a maximum matching. Thus we may assume that $|Y|\leq k$.
Note that every $s$-clique in $\hh$ contains at least $s-r+1$ vertices in $S$. We  proceed to derive an upper bound on the number of $s$-cliques in $\hh$. First, we choose a set $A$ of $(s-r+1)$ vertices in $S$.  There are at most ${|S|\choose s-r+1}$ {choices} for $A$. Then choose an $(r-1)$-element subset $B$ of $V(\hh)$ such that $\hh[B\cup A]$ is an $s$-clique of $\hh$. It  can be seen that $B$ is a common neighbor of  the vertices in $A$, that is, $B\in N_{\hh}(v)$ for any $v\in A$. If $A\subset Y$, then the number of choices for $B$ is at most $rk{n-2\choose r-2}$. If there exists a vertex $x\in A\setminus Y$, then the number of choices for $B$ is at most $2k{n-2\choose r-2}$. Hence
\[
K_s^r(\hh)\leq kr{n-2\choose r-2}{|Y|\choose s-r+1}+2k{n-2\choose r-2}{|S|\choose s-r+1}.
\]
Since $|Y|\leq k$ and $|S|=kr$, we find that
\begin{align}\label{eqs13}
K_s^r(\hh)\leq kr{n-2\choose r-2}{k\choose s-r+1}+2k{n-2\choose r-2}{rk\choose s-r+1}.
\end{align}
Using the inequality \eqref{eq4}, we see that
\begin{align}\label{eqs14}
{rk\choose s-r+1} \leq (er)^{s-r+1}{k\choose s-r+1}.
\end{align}
Combining \eqref{eqs13} and \eqref{eqs14}, we obtain that
\begin{align}\label{eqs15}
K_s^r(\hh)&\leq kr{k\choose s-r+1}{n-2\choose r-2}+2k(er)^{s-r+1}{k\choose s-r+1}{n-2\choose r-2}\nonumber\\[5pt]
&\leq \left(2k \left(er\right)^{s-r+1}+rk\right){k\choose s-r+1}{n-2\choose r-2}\nonumber\\[5pt]
&\leq 3k \left(er\right)^{s-r+1}{k\choose s-r+1}{n-2\choose r-2}.
\end{align}
Employing \eqref{eq3} and \eqref{eq5}, we find that
\begin{align}\label{eqs16}
{n-2\choose r-2}&\leq \left(\frac{n-2-(r-2)}{(n-k-1)-(r-2)}\right)^{r-2}{n-k-1\choose r-2}\nonumber\\[5pt]
&=\left(1+\frac{k-1}{n-k-r+1}\right)^{r-2}{n-k-1\choose r-2}\nonumber\\[5pt]
&\leq \left(1+\frac{(r-2)^2(k-1)}{n-k-r+1}\right){n-k-1\choose r-2}.
\end{align}
It follows from \eqref{eqs15} and \eqref{eqs16} that
\begin{align}\label{eqs17}
K_s^r(\hh)\leq 3(er)^{s-r+1}k {k\choose s-r+1} \left(1+\frac{(r-2)^2(k-1)}{n-k-r+1}\right){n-k-1\choose r-2}.
\end{align}
Since $n\geq 4(er)^{s-r+2}k$, we see that
\begin{align}\label{eqs18}
\frac{(r-2)^2(k-1)}{n-k-r+1}\leq \frac{1}{3}
\end{align}
and
\begin{align}\label{eqs19}
4(er)^{s-r+1}k\cdot \frac{r-1}{n-k} \leq 1.
\end{align}
In view of \eqref{eqs17}, \eqref{eqs18} and \eqref{eqs19}, we arrive at
\begin{align*}
K_s^r(\hh) \leq{k\choose s-r+1}{n-k\choose r-1} \leq K_s^{r}(\mathcal{F}^{(r)}_{n,k,1}).
\end{align*}
This completes the proof.
\end{proof}

\section{The maximum number of $s$-cliques with $s\geq k+r$}

 In this section, we prove parts (II) and (III) of Theorem \ref{th1} by utilizing the shifting method originally due to Erd\H{o}s-Ko-Rado \cite{ekr} and further developed by Frankl \cite{fra-shift}.

Let $\hh$ be an $r$-graph on the vertex set $[n]$. For integers $i,j$ with $1\leq i< j\leq n$ and any $E\in \hh$, the shifting operator $S_{ij}$ is defined by
 \begin{align*}
 S_{ij}(E)=\left\{
             \begin{array}{ll}
               (E\setminus\{j\})\cup\{i\}, & \hbox{if } j\in E, i\notin E \hbox{ and } (E\setminus\{j\})\cup\{i\}\notin \hh; \\[5pt]
               E, & \hbox{otherwise.}
             \end{array}
           \right.
 \end{align*}
Set
\[
S_{ij}(\hh)=\{S_{ij}(E): E\in H\}.
\]
An $r$-graph $\hh$ is called a stable $r$-graph if $S_{ij}(\hh)=\hh$ holds for all $1\leq i<j\leq n$, see Frankl \cite{fra-shift}. He showed that any $r$-graph $\hh$ can be shifted to a stable $r$-graph by applying the shifting operator iteratively.

We aim to determine the maximum number of  $s$-cliques in an $r$-graph $\hh$ with matching number at most $k$. Frankl \cite{fra-shift} proved that $\nu(S_{ij}(\hh))\leq \nu(\hh)$. Thus, the shifting operator preserves the property that the matching number is at most $k$. We shall show that the shifting operator does not decrease the number of $s$-cliques in an $r$-graph.

\begin{lem}\label{le4}
Let $\hh$ be an $r$-graph on the vertex set $[n]$. For any $i,j\in[n]$ with $i< j$ and $s\geq r$, we have $K_s^r(S_{ij}(\hh))\geq K_s^r(\hh)$. Moreover, if each edge of $\hh$ is contained in an $s$-clique of $\hh$, then each edge of $S_{ij}(\hh)$ is contained in an $s$-clique of $S_{ij}(\hh)$.
\end{lem}

\begin{proof}
Let $K\subset[n]$ with $|K|=s$. If $\hh[K]$ is an $s$-clique but $S_{ij}(\hh)[K]$ is not an $s$-clique, then we have $j\in K$ and $i\notin K$ and there is an edge in $\hh[K]$ that is shifted by $S_{ij}$.   By the definition of the shifting operator, it follows that $\hh[(K-\{j\})\cup \{i\}]$ is not an $s$-clique but $S_{ij}(\hh)[(K\setminus\{j\})\cup \{i\}]$ is an $s$-clique. Now, we define a map $\sigma$ from $\mathcal{K}_s^r(\hh)$ to $\mathcal{K}_s^r(S_{ij}(\hh))$ as follows. If $\hh[K]\in \mathcal{K}_s^r(\hh)$ and $S_{ij}(\hh)[K]\in \mathcal{K}_s^r(S_{ij}(\hh))$, let $\sigma(\hh[K]) = S_{ij}(\hh)[K]$; If $\hh[K]\in \mathcal{K}_s^r(\hh)$ but $S_{ij}(\hh)[K]\notin \mathcal{K}_s^r(S_{ij}(\hh))$, let $\sigma(\hh[K]) = S_{ij}(\hh)[(K-\{j\})\cup \{i\}]$.  It is easy to verify that $\sigma$ is an injection, and so $K_s^r(S_{ij}(\hh))\geq K_s^r(\hh)$.

Suppose that each edge of $\hh$ is contained in an $s$-clique of $\hh$ and there exists an edge $E\in S_{ij}(\hh)$ that is not contained in any $s$-clique of $S_{ij}(\hh)$.
 If $E\in \hh$, let $\hh[K]$ be an $s$-clique of $\hh$ containing $E$, where $K$ is a subset of $[n]$ with $|K|=s$.
 Since $E\in S_{ij}(\hh)$, we have $S_{ij}(\hh)[(K\setminus\{j\})\cup \{i\}]$ is an $s$-clique containing $E$, a contradiction.
 If $E\notin \hh$, then $E'=(E\setminus\{i\})\cup \{j\}$ is an edge of $\hh$.
 Let $K$ be a subset of $[n]$ such that $\hh[K]$ is an $s$-clique in $\hh$ containing $E'$.
 Clearly, we have $j\in E'$ and $i\notin K$.
 Then $S_{ij}(\hh)[(K\setminus\{j\})\cup \{i\}]$ is an $s$-clique in $S_{ij}(\hh)$ containing $E$, a contradiction.
 Thus, if each edge of $\hh$ is contained in an $s$-clique of $\hh$, then each edge of $S_{ij}(\hh)$ is contained in an $s$-clique of $S_{ij}(\hh)$. This completes the proof.
\end{proof}

Let us recall a basic property of the shifting operator, which will be used later.
Let $E_1=\{a_1,a_2,\ldots,a_r\}$ and $E_2=\{b_1,b_2,\ldots,b_r\}$ be two different $r$-element subsets of $[n]$.  As used in \cite{frankl4}, we write $E_1\prec E_2$ if there exists a permutation $\sigma_1\sigma_2\cdots\sigma_r$ of $[r]$ such that $a_j\leq b_{\sigma_j}$ for all $j=1,\ldots,r$. Frankl \cite{fra-shift} showed that if $\hh$ is a stable $r$-graph on $[n]$, $E\in \hh$ and $S$ is an $r$-element subset of $[n]$ with $S\prec E$, then $S\in \hh$.

The following proposition gives a characterization of  stable $r$-graphs with matching number at most $k$.


\begin{prop}\label{prop1}
Let $n, k, r, s$ be positive integers with $k+r\leq s\leq rk+r-1$ and $n\geq rk+r-1$. Let $\hh$ be a stable $r$-graph on the vertex set $[n]$ with $\nu(\hh)\leq k$. If every edge of $\hh$ is contained in at least one $s$-clique in $\hh$, then for every edge $E\in \hh$, we have $|E\cap[rk+a-1]|\geq a$, where $a=\lfloor \frac{s-r}{k} \rfloor+1$.
\end{prop}

\begin{proof}
 It is easily checked that $2\leq a\leq r$ and $(a-1)k+r\leq s\leq ak+r-1$. Suppose that there is an edge $E=\{x_1,x_2,\ldots,x_r\}\in \hh$ with $x_1<x_2<\cdots<x_r$ such that $|E\cap[rk+a-1]|< a$. Then $|E\cap[rk+a-1]|< a$ implies that $x_a\geq rk+a$. Let $K$ be an $s$-clique in $\hh$ containing $E$ and $X=\{x_a,x_{a+1},\ldots,x_r\}$. Since
\[
|V(K)\setminus X|=s-(r-a+1)\geq (a-1)k+r-(r-a+1)= (a-1)(k+1),
\]
 there exist $k+1$ disjoint $(a-1)$-element sets $S_1, S_2, \ldots, S_{k+1}$ in $V(K)\setminus X$.
 Moreover, $S_i\cup X$ is an edge of $\hh$ for each $i=1,2,\ldots, k+1$.  For any
 \[
 T\subset [rk+a-1]\setminus(\cup_{i=1}^{k+1} S_i)
 \]
  with $|T|=r-a+1$, $S_i\cup T$ forms an edge of $\hh$ for each $i=1,2,\ldots, k+1$, since $\hh$ is stable and $S_i\cup T\prec S_i\cup X$.
 Noting that
 \[
 |[rk+a-1]\setminus(\cup_{i=1}^{k+1} S_i)|\geq rk+a-1 -(a-1)(k+1) = (r-a+1)k,
 \]
  there are $k$ disjoint $(r-a+1)$-element sets $T_1, T_2, \ldots, T_{k}$ in $[rk+a-1]\setminus(\cup_{i=1}^{k+1} S_i)$. Thus, $S_1\cup T_1$, $S_2\cup T_2$, \ldots, $S_k\cup T_k$, $S_{k+1}\cup X$ constitute $k+1$ disjoint edges in $\hh$, which contradicts the fact that $\nu(\hh)\leq k$. This completes the proof.
\end{proof}

Moreover, we need a result similar to Lemma \ref{le1}, which can be proved by the greedy {algorithm}.

\begin{lem}\label{le7}
Let $n,k,r,a$ be integers such that $rk\leq n$, $a<r$ and $\hh$ be an $r$-graph on $n$ vertices. If $V(\hh)$ has $k$ disjoint $a$-element subset $A_1, A_2, \ldots, A_k$ with  $\mathrm{deg}(A_i)> r(k-1){n-a-2\choose r-a-1}$, then $\hh$ contains a matching of size $k$.
\end{lem}

\begin{proof}
Using the greedy algorithm, we can find a matching of size $k$ in $\hh$. Since
\[
|N_{\hh}(A_1)|> r(k-1){n-a-1\choose r-a-1} > |\cup_{j=2}^{k}A_j|{n-a-1\choose r-a-1},
\]
we can choose $B_1$ from $N_{\hh}(A_1)$ such that $B_1$ is disjoint from $\cup_{j=2}^{k}A_j$. For $i\in \{2,\ldots,k\}$, suppose that $B_1, B_2,\ldots, B_{i-1}$ have been chosen such that $A_1\cup B_1$, $A_2\cup B_2$, $\ldots$, $A_{i-1}\cup B_{i-1}$, $A_i$, $A_{i+1}$, $\ldots$, $A_k$ are pairwise disjoint. Since
\[
|N_{\hh}(A_i)|> r(k-1){n-a-1\choose r-a-1} > \left(\sum_{j=1}^{i-1}|A_j\cup B_j|+\sum_{j=i+1}^{k}|A_j|\right){n-a-1\choose r-a-1},
\]
we can choose $B_i$ from $N_{\hh}(A_i)$ such that $B_i$ is disjoint from $(\cup_{j=1}^k A_j)\cup(\cup_{j=1}^{i-1}B_j)$. Finally, we end up with a matching of size $k$ in $\hh$.
\end{proof}

 We first prove Theorem \ref{th1} (III), because it will be needed in the proof of Theorem \ref{th1} (II).

\begin{proof}[Proof of Theorem \ref{th1} (III)]
 Let $\hh^*$ be the subhypergraph obtained from $\hh$ by deleting all the edges in $\hh$ that are not contained in any $s$-clique in $\hh$. Frankl \cite{fra-shift} proved that any $r$-graph can be shifted to a stable $r$-graph by applying the shifting operator iteratively.
 By Lemma \ref{le4}, we may assume that $\hh^*$ is stable. Since $\lfloor\frac{s-r}{k}\rfloor+1=r$, by Proposition \ref{prop1} we obtain that $|E\cap [rk+r-1]|\geq r$ for every $E\in \hh^*$. So $\hh^*$ is a subhypergraph of $\mathcal{F}^{(r)}_{n,k,r}$, and thus
 \[K_s^r(\hh)\leq K_s^r(\hh^*)\leq K_s^r(\mathcal{F}^{(r)}_{n,k,r}).\] If $K_s^r(\hh)< K_s^r(\mathcal{F}^{(r)}_{n,k,r})$, then $\hh^*$ has to be a proper subhypergraph of $\mathcal{F}^{(r)}_{n,k,r}$. {It follows that} there is an $r$-element subset $T$ of $[rk+r-1]$ such that $T\notin \hh^*$. Then none of the $s$-element subsets of $[rk+r-1]$ containing $T$ can be {an} $s$-clique of $\hh^*$. Note that there are exactly ${rk-1\choose s-r}$  $s$-element subsets of $[rk+r-1]$ containing $T$. {Thus,}
\[
K_s^r(\hh)\leq K_s^r(\hh^*) \leq {rk+r-1\choose s}-{rk-1\choose s-r},
\]
as claimed.
\end{proof}

We are ready to prove Theorem \ref{th1} (II).

\begin{proof}[Proof of Theorem \ref{th1} (II)]
By Lemma \ref{le4}, we may assume that $\hh$ is a stable $r$-graph on $[n]$ and each edge of $\hh$ is contained in an $s$-clique. Note that $a=\lfloor \frac{s-r}{k} \rfloor+1$, so that $(a-1)k+r\leq s\leq ak+r-1$. By Proposition \ref{prop1}, we have $|E\cap[rk+a-1]|\geq a$ for every $E\in \hh$. Define an $a$-graph $\hh^*$ on $[rk+a-1]$ as
\[
\hh^*=\left\{A\in \binom{[rk+a-1]}{a}\colon \deg_{\hh}(A)> rk{n-a-1\choose r-a-1}\right\}.
\]
Now we prove the following two claims, leading to a description  of $\hh^*$.

{\noindent Claim 1.} $\hh^*$ is  stable.

Suppose to the contrary that $\hh^*$ is not  stable. Then, there exist  $i$ and $j$ such that $1\leq i<j\leq n$ and $S_{ij}(\hh^*) \neq \hh^*$. This ensures the existence of an edge $A\in \hh^*$ such that $S_{ij}(A) \neq A$.
By the definition of $\hh^*$, we have $|N_{\hh}(A)|=\deg_{\hh}(A)> rk{n-a-1\choose r-a-1}$. Let $A'=(A\setminus \{j\})\cup \{i\}$. Since $S_{ij}(A) \neq A$, we find that $j\in A$, $i\notin A$ and $A'\notin \hh^*$. Let $B\in N_{\hh}(A)$. If $i\notin B$, since $A\cup B\in \hh$ and $\hh$ is stable, it follows that $A'\cup B\in \hh$. If $i\in B$, since $A\cup B\in \hh$, we see that $A'\cup (B\setminus\{i\})\cup\{j\}=A\cup B\in \hh$. Now we define a map $\tau$ from $N_{\hh}(A)$ to $N_{\hh}(A')$. If $i\notin B$, let $\tau(B)=B$; if $i\in B$, let $\tau(B)=(B\setminus\{i\})\cup\{j\}$. It can be seen that $\tau$ is injective and $|N_{\hh}(A')|\geq |N_{\hh}(A)|> rk{n-a-1\choose r-a-1}$, which contradicts the fact that $A'\notin \hh^*$. This proves the claim.

{\noindent Claim 2.} $\nu(\hh^*)\leq k$.

Suppose to the contrary that $\nu(\hh^*)\geq k+1$. Then, there exist $k+1$ disjoint edges $A_1,A_2,\ldots,A_{k+1}$ in $\hh^*$. Since $\deg_{\hh}(A_i)\geq rk{n-a-1\choose r-a-1}$ for each $i=1,2,\ldots,k+1$, by Lemma \ref{le7} there exists a matching of size $k+1$ in $\hh$, which contradicts the fact {that} $\nu(\hh)\leq k$. {Thus the claim holds.}

Since $|E\cap[rk+a-1]|\geq a$ for every edge $E\in \hh$, every $s$-clique in $\hh$ has at least $s-r+a$ vertices in $[rk-a+1]$. Now we consider the maximum number of $(s-r+a)$-cliques in $\hh^*$. Since $\hh^*$ is an $a$-graph and $(a-1)k+a\leq s-r+a\leq ak+a-1$, by Theorem \ref{th1} (III) we have \[
K_{s-r+a}^a(\hh^*)\leq K_{s-r+a}^a(\mathcal{F}^{(a)}_{n,k,a}) =\binom{ak+a-1}{s-r+a}.\] Moreover, if $K_{s-r +a}^a(\hh^*)< K_{s-r+a}^a(\mathcal{F}^{(a)}_{n,k,a})$, we have \[K_{s-r+a}^a(\hh^*)\leq \binom{ak+a-1}{s-r+a}-\binom{ak-1}{s-r}.\] Next we consider two cases depending on the value of $K_{s-r+a}^a(\hh^*)$.

Case 1. $K_{s-r+a}^a(\hh^*)= K_{s-r+a}^a(\mathcal{F}^{(a)}_{n,k,a})$. If there {exists} an edge $E\in E(\hh)$ with $|E\cap [ak+a-1]|\leq a-1$, then there are $k$ disjoint edges $A_1,A_2,\ldots,A_k$ in $\hh^*-E$. {Noting that there are at most $r{n-a-1\choose r-a-1}$ sets $T$ in $N_{\hh}(A_i)$ such that $|T\cap E|\geq 1$ and that $\deg_{\hh}(A_i)> rk{n-a-1\choose r-a-1}$  for each $i=1,2,\ldots,s-r+1$, we have}
\begin{align*}
\deg_{\hh- E}(A_i)&> rk{n-a-1\choose r-a-1}-r{n-a-1\choose r-a-1}\\[5pt]
&\geq r(k-1){n-r-a-1\choose r-a-1}.
\end{align*}
By Lemma \ref{le7}, there is a matching $\hm$ of size $k$ in $\hh-E$.
 Then $\hm\cup \{E\}$ forms a matching of size $k+1$ in $\hh$, which contradicts the fact that $\nu(\hh)\leq k$. Thus, $\hh$ is a subhypergraph of $\mathcal{F}^{(r)}_{n,k,a}$ and so $K_s^r(\hh)\leq K_s^r(\mathcal{F}^{(r)}_{n,k,a})$.

Case 2. $K_{s-r+a}^a(\hh^*)\leq \binom{ak+a-1}{s-r+a}-\binom{ak-1}{s-r}$. By Proposition \ref{prop1}, we have  $|E\cap[rk+a-1]|\geq a$  for any $E\in \hh$.
Thus, for each $s$-clique $K$ in $\hh$, $|V(K)\cap[rk+a-1]|\geq s-r+a$. We  aim to derive an upper bound on the number of $s$-cliques in $\hh$. First, we choose an $(s-r+a)$-element subset $S$ of $[rk+a-1]$. Then choose an $(r-a)$-element subset $T$ such that $\hh[S\cup T]$ forms an $s$-clique of $\hh$.  It can be seen that  $T$ is a common neighbor of the $a$-element subsets of $S$, that is , $T\in N_{\hh}(R)$ for any $R\in \binom{S}{a}$. {If $\hh^*[S]$ is not an $(s-r+a)$-clique in $\hh^*$}, there exists an $a$-element subset $A$ of $S$ such that $A\notin \hh^*$.
Hence the number of choices for $T$ is at most $\deg_{\hh}(A)\leq rk{n-a-1\choose r-a-1}$. If $\hh^*[S]$ is an $(s-r+a)$-clique in $\hh^*$, then the number of choices for $T$ is at most $\binom{n-(s-r+a)}{r-a}$. Therefore,
\begin{align}\label{eqs3}
K_s^r(\hh)\leq &K_{s-r+a}^a(\hh^*){n-s+r-a\choose r-a}+rk{n-a-1\choose r-a-1}{rk+a-1\choose s-r+a}\nonumber\\[5pt]
\leq&\left({ak+a-1\choose s-r+a}-{ak-1\choose s-r}\right){n-s+r-a\choose r-a}\nonumber\\[5pt] &\qquad\qquad+rk{n-a-1\choose r-a-1}{rk+a-1\choose s-r+a}.
\end{align}

If $s=ak+r-1$, substituting $s=ak+r-1$ into \eqref{eqs3}, we obtain that
\begin{align}\label{eqs21}
K_s^r(\hh)\leq rk{n-a-1\choose r-a-1}{rk+a-1\choose ak+a-1}.
\end{align}
It follows from \eqref{eq3} and \eqref{eq5} that
\begin{align}\label{eqs20}
{n-a-1\choose r-a-1}&\leq \left(\frac{n-r}{n-ak-r+1}\right)^{r-a-1}{n-ak-a\choose r-a-1}\nonumber\\[5pt]
&\leq \left(1+\frac{(ak-1)(r-a-1)^{2}}{n-ak-r+1}\right){n-ak-a\choose r-a-1}.
\end{align}
Since $n\geq4r^2(er/a)^{s-r+a}k$, we see that
\begin{align}\label{eqsa21}
\frac{(ak-1)(r-a-1)^{2}}{n-ak-r+1} \leq 1.
\end{align}
By \eqref{eqs20} and \eqref{eqsa21}, we get
\begin{align}\label{eqs2}
{n-a-1\choose r-a-1}\leq 2{n-ak-a\choose r-a-1}.
\end{align}
Combining \eqref{eqs21} and \eqref{eqs2}, we obtain that
\begin{align}\label{eqs22}
K_s^r(\hh)\leq\ 2rk{rk+a-1\choose ak+a-1}{n-ak-a\choose r-a-1}.
\end{align}
Applying \eqref{eq1} and \eqref{eqs22} gives
\begin{align}\label{eqs23}
K_s^r(\hh)&\leq 2rk\left(\frac{e(rk+a-1)}{ak+a-1}\right)^{ak+a-1}{n-ak-a\choose r-a-1}\nonumber\\[5pt]
&\leq  2kr \left(\frac{er}{a}\right)^{ak+a-1} \frac{r-a}{n-ak-a+1}{n-ak-a+1\choose r-a}.
\end{align}
Under the condition $s=ak+r-1$, we have $ak+a-1=s-r+a$. Since $n\geq 4r^2(er/a)^{s-r+a}k$, it can be checked that
\begin{align}\label{eqs24}
2kr \left(\frac{er}{a}\right)^{ak+a-1} \frac{r-a}{n-ak-a+1} \leq 1.
\end{align}
Combining \eqref{eqs23} and \eqref{eqs24}, we get
\[
K_s^r(\hh)\leq {n-ak-a+1\choose r-a} = K_s^r(\mathcal{F}^{(r)}_{n,k,a}).
\]

It remains to consider the case $(a-1)k+r\leq s<ak+r-1$.  We have
\begin{align}\label{eq7}
{ak+a-1\choose s-r+a}= &\frac{ak+a-1}{s-r+a}\cdot\frac{ak+a-2}{s-r+a-1} \cdots  \frac{ak}{s-r+1}\cdot{ak-1\choose s-r}\nonumber \\[5pt]
\leq &\left(\frac{ak-1}{s-r}\right)^a{ak-1\choose s-r}\nonumber\\[5pt]
\leq &\left(\frac{ak-1}{(a-1)k}\right)^a{ak-1\choose s-r}\nonumber\\[5pt]
\leq &\left(\frac{a}{a-1}\right)^a{ak-1\choose s-r}.
\end{align}
Employing \eqref{eq4} and \eqref{eq7}, we find that
\begin{align}\label{eq8}
{rk+a-1\choose s-r+a}&\leq \left(\frac{e(rk+a-1)}{ak+a-1}\right)^{s-r+a}{ak+a-1\choose s-r+a}\nonumber\\[5pt]
&\leq \left(\frac{er}{a}\right)^{s-r+a}\left(\frac{a}{a-1}\right)^a{ak-1\choose s-r}\nonumber\\[5pt]
&\leq \left(\frac{er}{a-1}\right)^{s-r+a}{ak-1\choose s-r}.
\end{align}
It follows from \eqref{eq3} and \eqref{eq5} that
\begin{align}\label{eqs30}
{n-s+r-a\choose r-a}&\leq \left(\frac{n-s+r-a-(r-a)}{n-ak-a+1-(r-a)}\right)^{r-a}{n-ak-a+1\choose r-a}\nonumber\\[5pt]
&=\left(1+\frac{ak+r-1-s}{n-ak-r+1}\right)^{r-a}{n-ak-a+1\choose r-a}\nonumber\\[5pt]
&\leq \left(1+\frac{(r-a)^2 (ak+r-1-s)}{n-ak-r+1}\right){n-ak-a+1\choose r-a}.
\end{align}
 The condition $(a-1)k+r\leq s< ak+r-1$  implies $ak+r-1-s\leq k$. Since $n\geq 2(ak+r-1)$,  from \eqref{eqs30} we see that
\begin{align}\label{eq9}
{n-s+r-a\choose r-a}
&\leq \left(1+\frac{(r-a)^{2}k}{n-ak-r+1}\right){n-ak-a+1\choose r-a}\nonumber\\[5pt]
&\leq \left(1+\frac{2r^2k}{n}\right){n-ak-a+1\choose r-a}.
\end{align}
Substituting \eqref{eq9} into \eqref{eqs3}, we obtain that
\begin{align}\label{eqs25}
K_s^r(\hh)
&\leq\left({ak+a-1\choose s-r+a}-{ak-1\choose s-r}\right)\left(1+\frac{2r^2k}{n}\right){n-ak-a+1\choose r-a}\nonumber\\[5pt]
&\qquad\qquad+rk{n-a-1\choose r-a-1}{rk+a-1\choose s-r+a}\nonumber\\[5pt]
&\leq \left({ak+a-1\choose s-r+a}-{ak-1\choose s-r}\right){n-ak-a+1\choose r-a}+\frac{2r^2k}{n}{ak+a-1\choose s-r+a}\nonumber\\[5pt]
&\qquad\qquad \cdot{n-ak-a+1\choose r-a}+rk{n-a-1\choose r-a-1}{rk+a-1\choose s-r+a}.
\end{align}
Since each edge of $\mathcal{F}^{(r)}_{n,k,a}$ has at least $a$ vertices in $[ak+a-1]$, it is easy to see that each $s$-clique of $\mathcal{F}^{(r)}_{n,k,a}$ has at least $s-r+a$ vertices in $[ak+a-1]$. Thus,
\begin{align}\label{eqs26}
K_s^r(\mathcal{F}^{(r)}_{n,k,a})= \sum_{i=s-r+a}^s {ak+a-1\choose i}{n-ak-a+1\choose s-i}> {ak+a-1\choose s-r+a}{n-ak-a+1\choose r-a}.
\end{align}
Using the inequalities \eqref{eqs25} and \eqref{eqs26}, we find that
\begin{align}\label{eqs27}
K_s^r(\hh)&<  K_s^r(\mathcal{F}^{(r)}_{n,k,a})-{ak-1\choose s-r}{n-ak-a+1\choose r-a}+\frac{2r^2k}{n}{ak+a-1\choose s-r+a}{n-ak-a+1\choose r-a}\nonumber\\[5pt]
&\qquad\qquad+rk {n-a-1\choose r-a-1}{rk+a-1\choose s-r+a}.
\end{align}
Substituting  (\ref{eq7}), (\ref{eq8}), (\ref{eqs2}) into \eqref{eqs27}, we deduce that
\begin{align}\label{eqs28}
K_s^r(\hh)&\leq  K_s^r(\mathcal{F}^{(r)}_{n,k,a})-{ak-1\choose s-r}{n-ak-a+1\choose r-a}+\frac{2r^2k}{n}\left(\frac{a}{a-1}\right)^a{ak-1\choose s-r}\nonumber\\[5pt]
&\qquad\qquad \cdot{n-ak-a+1\choose r-a}+\left(\frac{er}{a-1}\right)^{s-r+a}{ak-1\choose s-r}\cdot rk\cdot 2 {n-ak-a\choose r-a-1}\nonumber\\[5pt]
&= K_s^r(\mathcal{F}^{(r)}_{n,k,a})-{ak-1\choose s-r}{n-ak-a+1\choose r-a}\nonumber\\[5pt]
&\qquad\qquad \cdot \left(1-\left(\frac{a}{a-1}\right)^a\frac{2r^2k}{n}- \left(\frac{er}{a-1}\right)^{s-r+a}\frac{2rk(r-a)}{n-ak-a+1}\right)
\end{align}
Since $n\geq 4r^2k(er/(a-1))^{s-r+a}$, we obtain that
\begin{align}\label{eqs29}
1-\left(\frac{a}{a-1}\right)^a\frac{2r^2k}{n}- \left(\frac{er}{a-1}\right)^{s-r+a}\frac{2rk(r-a)}{n-ak-a+1} >0.
\end{align}
Combining \eqref{eqs28} and \eqref{eqs29}, we arrive at $K_s^r(\hh)\leq K_s^r(\mathcal{F}^{(r)}_{n,k,a})$. This completes the proof.
\end{proof}

\section{Proof of Theorem \ref{th3}}

Huang, Loh and Sudakov \cite{huang} considered a multicolored generalization of the Erd\H{o}s matching conjecture and provided a sufficient condition on the number of edges for a multicolored hypergraph to contain a rainbow matching of size $k$, as stated in Lemma \ref{le3} below. Theorem 1.6 can be considered as a generalization  of Theorem 1.5. The proof of Theorem \ref{th3} also relies on Lemma \ref{le3}.



Let $\mathcal{F}_1, \mathcal{F}_2, \ldots, \mathcal{F}_{k}$ be $r$-graphs on $[n]$. We say that $\{\mathcal{F}_1, \mathcal{F}_2, \ldots, \mathcal{F}_{k}\}$ contains a rainbow matching if there exist $k$ pairwise disjoint sets $F_1\in \mathcal{F}_1, F_2\in\mathcal{F}_2, \ldots, F_k\in\mathcal{F}_k$.

\begin{lem}[Huang, Loh and Sudakov \cite{huang}]\label{le3}
Let $\mathcal{F}_1, \mathcal{F}_2, \ldots, \mathcal{F}_k$ be $r$-graphs on $[n]$ such that $|\mathcal{F}_i|>(k-1){n-1\choose r-1}$, and $n\geq rk$. Then $\{\mathcal{F}_1, \mathcal{F}_2, \ldots, \mathcal{F}_k\}$ contains a rainbow matching.
\end{lem}

Theorem \ref{th3}  will be proved by  induction. The following lemma is
the basis of the induction.

\begin{lem}\label{le6}
Let $n$, $k$ and $r$ be integers  such that $n\geq  4 k^2(er)^{k}$. Let $\mathcal{F}_1, \mathcal{F}_2, \ldots, \mathcal{F}_{k}$ be $r$-graphs on  $[n]$. If for all $i\in \{1,2,\ldots,k\}$, there exists some $s\in \{r,r+1,\ldots,k+r-2\}$ such that $K_s^r(\mathcal{F}_i)>  K_s^r(\mathcal{F}^{(r)}_{n,k-1,1})$.
Then the family $\{\mathcal{F}_1, \mathcal{F}_2, \ldots, \mathcal{F}_{k}\}$ contains a rainbow matching.
\end{lem}

\begin{proof}
Let $\{\mathcal{F}_1, \mathcal{F}_2, \ldots, \mathcal{F}_{k}\}$ be a family of $r$-graphs that does not contain any rainbow matching. We may further assume that this family attains the maximum value of $\sum_{i=1}^k |\mathcal{F}_i|$. We shall prove the lemma by showing that there exists some $i$ such that $$K_s^r(\mathcal{F}_i) \leq K_s^r(\mathcal{F}^{(r)}_{n,k-1,1})$$ for any $s\in \{r,r+1,\ldots,k+r-2\}$.

Let $l$ be the number of $r$-graphs in the family $\{\mathcal{F}_1, \mathcal{F}_2, \ldots, \mathcal{F}_{k}\}$ that are not complete $r$-graphs.
Without loss of generality, we may assume that $\mathcal{F}_{1}, \ldots, \mathcal{F}_{l}$ are such non-complete $r$-graphs and $\mathcal{F}_{l+1}, \ldots, \mathcal{F}_{k}$ are complete $r$-graphs. For $l=1$, if $\mathcal{F}_1$ is not an empty $r$-graph, by the definition of $l$, there exist disjoint edges $F_1\in \mathcal{F}_1, F_2\in \mathcal{F}_2,\ldots, F_{k}\in \mathcal{F}_{k}$,  contradicting the assumption that $\{\mathcal{F}_1, \mathcal{F}_2, \ldots, \mathcal{F}_{k}\}$  does not contain any rainbow matching. If  $\mathcal{F}_1$ is  an empty $r$-graph, we have $$K_s^r(\mathcal{F}_1) =0\leq K_s^r(\mathcal{F}^{(r)}_{n,k-1,1})$$ for any $s\in \{r,r+1,\ldots,k+r-2\}$.
Thus, we may assume that $2\leq l\leq k$.

For  $i=1,2,\ldots,l$, let $X_i$ be the set of  vertices $v\in [n]$ such that $\deg_{\mathcal{F}_i}(v)> 2(l-1){n-2\choose r-2}$ and let $Y_i$ be the set of  vertices $v\in [n]$ such that $\deg_{\mathcal{F}_i}(v)\geq r(k-1){n-2\choose r-2}$. It is clear that $Y_i\subseteq X_i$.

{\noindent Claim 3.}
The family $\{X_1, X_2, \ldots, X_l\}$ does not contain a system of distinct representatives.

Suppose to the contrary that there exists a system of distinct representatives in $\{X_1, X_2, \ldots, X_l\}$. Assume that $x_1\in X_1, x_2\in X_2, \ldots, x_l\in X_l$ are $l$ distinct vertices. Let $X=\{x_1,x_2,\ldots,x_l\}$. For $i=1,2,\ldots,l$, define
\[
\hh_i = \left\{T\in \binom{[n]\setminus X}{r-1}\colon T\cup \{x_i\} \in \mathcal{F}_i \right\}.
\]
For any  $i,j\in [l]$ with $i\neq j$, there are at most ${n-2\choose r-2}$ edges of $\mathcal{F}_i$ containing both $x_i$ and $x_j$. Thus, for $i=1,2,\ldots,l$,
\[
|\hh_i|\geq \deg_{\mathcal{F}_i}(v_i)-(l-1){n-2\choose r-2}>(l-1){n-2\choose r-2}\geq (l-1){n-l-1\choose r-2}.
\]

 Since $\hh_i$ is an $(r-1)$-graph on $n-l$ vertices, by Lemma \ref{le3}, there exist $l$ disjoint edges $E_1\in \hh_1, E_2\in \hh_2, \ldots, E_l\in \hh_l$. It follows that $\{E_1\cup \{x_1\},E_2\cup \{x_2\},\ldots,E_l\cup \{x_l\}\}$ forms a rainbow matching in $\{\mathcal{F}_{1}, \ldots, \mathcal{F}_{l}\}$. Since $\mathcal{F}_{l+1}, \ldots, \mathcal{F}_{k}$ are all complete $r$-graphs, there exists a rainbow matching in $\{\mathcal{F}_{1}, \ldots, \mathcal{F}_{k}\}$, a contradiction. This proves the claim.

The following claim shows that if $|X_i|$ and $|Y_i|$ are both small, then the lemma follows.

{\noindent Claim 4.}
If there exists $i\in \{1,2,\ldots,l\}$ such that $|X_i|\leq l-1$ and $|Y_i|\leq l-2$, then $K_s^r(\mathcal{F}_i)\leq  K_s^r(\mathcal{F}^{(r)}_{n,k-1,1})$  for all $r\leq s\leq k+r-2$.

Since $\mathcal{F}_i$ is not a complete $r$-graph, by the maximality of $\sum_{i=1}^k |\mathcal{F}_i|$, we see that $\{\mathcal{F}_1, \ldots, \mathcal{F}_{i-1}, \mathcal{F}_{i+1},  \ldots, \mathcal{F}_{k}\}$ contains a rainbow matching. Let  $$\hm=\{E_1, \ldots, E_{i-1}, E_{i+1},\ldots, E_k\}$$ be such a rainbow matching and  $S$ be the set of vertices that are covered by $\hm$.  For each $s\in\{r, r+1,\ldots, k+r-2\}$, every $s$-clique in $\mathcal{F}_i$ has at least $s-r+1$ vertices in $S$.
To derive an upper bound on the number of $s$-cliques in $\mathcal{F}_i$,  we first choose an $(s-r+1)$-element subset $A$ of $S$.
 Then choose an $(r-1)$-element subset $B$ of $V$, such that $\mathcal{F}_i[B\cup A]$ is an $s$-clique of $\mathcal{F}_i$. It  can be seen that $B$ is a common neighbor in $\mathcal{F}_i$ of  the vertices in $A$, that is, $B\in N_{\mathcal{F}_i}(v)$ for any $v\in A$.

 If $A\subset Y_i$,  the number of choices for $B$ is at most $\binom{n-s+r-1}{r-1}$. If $A\subset X_i$ and $A\nsubseteq Y_i$,  the number of choices for $B$ is at most $r(k-1)\binom{n-2}{r-2}$. If $A\nsubseteq X_i$,  the number of choices for $B$ is at most $2(l-1)\binom{n-2}{r-2}$. Thus,
\begin{align*}
K_s^r(\mathcal{F}_i)&\leq {|Y_i|\choose s-r+1}{n-(s-r+1)\choose r-1}+r(k-1){n-2\choose r-2}{|X_i|\choose s-r+1} \\[5pt]
&\qquad\qquad+2(l-1){n-2\choose r-2}{|S|\choose s-r+1}.
\end{align*}
Since $|Y_i|\leq l-2\leq k-2$, $|X_i|\leq l-1\leq k-1$ and $|S|=r(k-1)$, we find that
\begin{align}\label{eq10}
K_s^r(\mathcal{F}_i)
&\leq {k-2\choose s-r+1}{n-s+r-1\choose r-1}+ r(k-1){n-2\choose r-2}{k-1\choose s-r+1}\nonumber\\[5pt]
&\qquad\qquad+2(k-1){n-2\choose r-2}{r(k-1)\choose s-r+1}.
\end{align}
If $s=k+r-2$, by the inequality \eqref{eq10} we have
\begin{align}\label{eqs4}
 K_s^r(\mathcal{F}_i) &\leq r(k-1){n-2\choose r-2} +2(k-1){n-2\choose r-2}{r(k-1)\choose k-1}\nonumber\\[5pt]
 &\leq 3(k-1){r(k-1)\choose k-1}{n-2\choose r-2}.
\end{align}
Using the inequality \eqref{eq1}, we get
\begin{align}\label{eqs31}
{r(k-1)\choose k-1} \leq (er)^{k-1}.
\end{align}
Employing \eqref{eq3} and \eqref{eq5}, we see that
\begin{align}\label{eqs32}
{n-2\choose r-2} &\leq \left(\frac{n-2-(r-2)}{n-k-(r-2)}\right)^{r-2}{n-k\choose r-2}\nonumber\\[5pt]
&\leq \left(1+\frac{(r-2)^2(k-2)}{n-k-r+2}\right) {n-k\choose r-2}.
\end{align}
 Substituting \eqref{eqs31} and \eqref{eqs32} into \eqref{eqs4}, we obtain that
\begin{align}\label{eqs33}
 K_s^r(\mathcal{F}_i) \leq 3k(er)^{k-1} \left(1+\frac{(r-2)^2(k-2)}{n-k-r+2}\right) \frac{r-1}{n-k+1}\cdot {n-k+1\choose r-1}.
\end{align}
Since $n\geq 4k^2(er)^k$, we have
\begin{align}\label{eqs34}
\frac{(r-2)^2(k-2)}{n-k-r+2} \leq 1
\end{align}
and
\begin{align}\label{eqs35}
\frac{3k(er)^{k}}{n-k+1} \leq 1.
\end{align}
In view of \eqref{eqs33}, \eqref{eqs34} and \eqref{eqs35}, we arrive at
\begin{align*}
K_s^r(\mathcal{F}_i) \leq {n-k+1\choose r-1}\leq K_s^r(\mathcal{F}^{(r)}_{n,k-1,1}).
\end{align*}
It remains to consider the case $r\leq s\leq k+r-3$. Applying \eqref{eq3} and \eqref{eq5} gives
\begin{align}\label{eqx2}
{n-s+r-1\choose r-1} \leq&\left(1+\frac{(k+r-2-s)}{n-k-r+2}\right)^{r-1} {n-k+1\choose r-1}\nonumber\\[5pt]
\leq&  \left(1+\frac{(r-1)^2(k+r-2-s)}{n-k-r+2}\right) {n-k+1\choose r-1}\nonumber\\[5pt]
\leq& \left(1+\frac{2r^2k}{n}\right) {n-k+1\choose r-1},
\end{align}
and \begin{align}\label{eqs45}
{n-2\choose r-2}\leq&\left(\frac{n-2-(r-2)}{n-k-(r-2)}\right)^{r-2}{n-k\choose r-2}\nonumber\\[5pt]
\leq &\left(1+\frac{(r-2)^2(k-2)}{n-k-r+2}\right){n-k\choose r-2}.
\end{align}
Combining \eqref{eq10} and \eqref{eqx2}, we deduce that
\begin{align}\label{eqs43}
K_s^r(\mathcal{F}_i)
&\leq \left(1+\frac{2r^2k}{n}\right){k-2\choose s-r+1} {n-k+1\choose r-1}+r(k-1){k-1\choose s-r+1}{n-2\choose r-2}\nonumber\\[5pt]
&\qquad\qquad+ 2(k-1){r(k-1)\choose s-r+1}{n-2\choose r-2}.
\end{align}
Using the inequality \eqref{eq1}, we get
\begin{align}\label{eqs44}
{r(k-1)\choose s-r+1}\leq (er)^{s-r+1}{k-1\choose s-r+1}.
\end{align}
Substituting \eqref{eqs45} and \eqref{eqs44} into \eqref{eqs43}, we obtain that

\begin{align}\label{eqs46}
K_s^r(\mathcal{F}_i)
&\leq \left(1+\frac{2r^2k}{n}\right){k-2\choose s-r+1} {n-k+1\choose r-1}\nonumber\\[5pt]
&\qquad\qquad+3(er)^{s-r+1}(k-1){k-1\choose s-r+1}\left(1+\frac{(r-2)^2(k-2)}{n-k-r+2}\right){n-k\choose r-2}.
\end{align}
Under the condition $n\geq  4k^2(er)^{k}$, we find that
\begin{align}\label{eqs49}
(k-1)(r-1)\left(1+\frac{(r-2)^2(k-2)}{n-k-r+2}\right)\leq erk.
\end{align}
It follows from \eqref{eqs46} and \eqref{eqs49} that
\begin{align}\label{eqs47}
K_s^r(\mathcal{F}_i)
\leq & \left({k-1\choose s-r+1}- {k-2\choose s-r}\right)\left(1+\frac{2r^2k}{n}\right) {n-k+1\choose r-1} +\frac{3(er)^{s-r+2}k}{n-k+1}\nonumber\\[5pt]
&\qquad\qquad \cdot {n-k+1\choose r-1}{k-1\choose s-r+1}\nonumber\\[5pt]
\leq & {k-1\choose s-r+1}{n-k+1\choose r-1}  \left(1+\frac{3(er)^{s-r+2}k}{n-k+1}+\frac{2r^2k}{n} -\frac{s-r+1}{k-1}\right).
\end{align}
Again, under the condition $n\geq  4k^2(er)^{k}$, we also have
\begin{align}\label{eqs48}
\frac{3(er)^{s-r+2}k}{n-k+1}+\frac{2r^2k}{n} -\frac{s-r+1}{k-1}<0.
\end{align}
Combining \eqref{eqs47} and \eqref{eqs48}, we obtain
\begin{align}
K_s^r(\mathcal{F}_i)\leq{k-1\choose s-r+1}{n-k+1\choose r-1}  \leq K_t^r(\mathcal{F}^{(r)}_{n,s-r+1,1}).
\end{align}
This complete the proof of Claim 4.

By Claim  3 and Hall's Marriage Theorem, there exists  $I\subset [l]$ such that $\left|\cup_{i\in I} X_i\right|<|I|$. By Claim 4, we only need to consider the case when $|X_i|\geq l$ or $|X_i|\geq |Y_i| \geq l-1$ for any $i=1,\ldots,l$. Thus, we may assume that $X_1=X_2=\cdots=X_l=Y_1=Y_2=\cdots=Y_l=\{x_1, x_2, \ldots, x_{l-1}\}$.

{\noindent Claim 5.}
For  $i\in\{1,2,\ldots,l\}$ and $E\in \mathcal{F}_i$, we have $E\cap\{x_1, x_2, \ldots, x_{l-1}\}\neq \emptyset$.

 We may assume that there exists $E\in \mathcal{F}_l$ such that $E\cap\{x_1, x_2, \ldots,x_{l-1}\}= \emptyset$. Since $\deg_{\mathcal{F}_i}(x_i)\geq r(k-1){n-2\choose r-2}$ for  $i=1,\ldots,l-1$,
 there exist disjoint edges $E_1\in \mathcal{F}_1,\ldots, E_{l-1}\in \mathcal{F}_{l-1}$ such that $(\bigcup_{i=1}^{l-1}E_i)\cap E=\emptyset$. Now $E_1,\ldots, E_{l-1},E$ forms a rainbow matching in $\{\mathcal{F}_1,\ldots,\mathcal{F}_{l-1},\mathcal{F}_{l}\}$. Since $\mathcal{F}_{l+1}, \ldots, \mathcal{F}_{k}$ are all complete $r$-graphs, one can find a rainbow matching in $\{\mathcal{F}_{1}, \ldots, \mathcal{F}_{k}\}$, a contradiction. This completes the proof of Claim 5.

Claim 5 implies that $\mathcal{F}_i$ is isomorphic to a subhypergraph of $\mathcal{F}^{(r)}_{n,l-1,1}$ for any $i=1,\ldots,l$. Consequently,
\[
K_s^r(\mathcal{F}_i)\leq K_s^r(\mathcal{F}^{(r)}_{n,l-1,1}) \leq K_s^r(\mathcal{F}^{(r)}_{n,k-1,1})
\]
for $r\leq s\leq k+r-2$. Therefore, this completes the proof.
\end{proof}

We are now in a position to  prove Theorem \ref{th3}. Notice that Theorem \ref{th3} is  implied by Lemma \ref{le6} for sufficiently large $n$. That is, when $n\geq 4k^2(er)^k$, if there exists $s\in \{r,r+1,\ldots,t\}$ such that $K_s^r(\mathcal{F}_i)>K_s^r(\mathcal{F}^{(r)}_{n,k-1,1})$ for $i=1,\ldots, k$,   then the family $\{\mathcal{F}_1, \mathcal{F}_2, \ldots, \mathcal{F}_{k}\}$ contains a rainbow matching. In the following proof of Theorem \ref{th3}, the lower bound on $n$ is improved to $n\geq4k(t-r+2)(er)^{t-r+1}$ for $t\leq k+r-2$.


\begin{proof}[Proof of Theorem \ref{th3}]
We proceed by induction on $k$. By Lemma \ref{le6}, the theorem holds for $k=t-r+2$ and $n\geq 4k(t-r+2)(er)^{t-r+2}$. Now  assume that the theorem holds for $k-1$.

Suppose that there exist  $v\in V$ and $i\in[k]$ such that $\{\mathcal{F}_1-v, \ldots, \mathcal{F}_{i-1}-v, \mathcal{F}_{i+1}-v,\ldots, \mathcal{F}_k-v\}$ does not contain any rainbow matching. By the induction hypothesis, there exists $j\in[k]\setminus\{i\}$ satisfying $K_s^r(\mathcal{F}_j-v)\leq K_s^r(\mathcal{F}^{(r)}_{n-1,k-2,1})$ for all $r\leq s\leq t$. For $s=r$, we have $$K_{r}^r(\mathcal{F}_j)\leq K_r^r(\mathcal{F}_j-v)+{n-1\choose r-1}\leq K_r^r(\mathcal{F}^{(r)}_{n,k-1,1}).$$
For $r+1\leq s\leq t$, by the equality (\ref{eq6}) we find that
\begin{align*}
K_{s}^r(\mathcal{F}_j)&=K_s^r(\mathcal{F}_j-v)+K_s^r(v,\mathcal{F}_j)\\[5pt]
&\leq K_s^r(\mathcal{F}_j-v)+K_{s-1}^r(\mathcal{F}_j-v)\\[5pt]
&\leq K_s^r(\mathcal{F}^{(r)}_{n-1,k-2,1})+K_{s-1}^r(\mathcal{F}^{(r)}_{n-1,k-2,1})\\[5pt]
&=K_s^r(\mathcal{F}^{(r)}_{n,k-1,1}).
\end{align*}

Suppose that for any $v\in V$ and  $i\in[k]$, $\{\mathcal{F}_1-v, \ldots, \mathcal{F}_{i-1}-v, \mathcal{F}_{i+1}-v,\ldots, \mathcal{F}_k-v\}$ contains a rainbow matching.
 This assumption implies that the maximum degree of each $\mathcal{F}_i$ is at most $r(k-1){n-2\choose r-2}$, otherwise we may find a rainbow matching  using the greedy algorithm.
For $i=1,2,\ldots,k$, let $X_i$ be the set of vertices $u\in V$ such that $d_{\mathcal{F}_i}(v)> 2(k-1){n-2\choose r-2}$.
 By the same argument as in  Claim 3 of Lemma \ref{le6}, we deduce that there exists  $j$ such that $|X_j|\leq k-1$. Let $\hm=\{E_1,\ldots,E_{j-1},E_{j+1},\ldots,E_k\}$ be a rainbow matching in $\{\mathcal{F}_1, \ldots, \mathcal{F}_{j-1}, \mathcal{F}_{j+1},\ldots, \mathcal{F}_k\}$ and let $S$ be the set of vertices covered by $\hm$.
 Note that each $s$-clique in $\mathcal{F}_{j}$ has at least $s-r+1$ vertices in $S$.

 We  wish to derive an upper bound on the number of $s$-cliques in $\mathcal{F}_{j}$. First, we choose a set $A$ of $(s-r+1)$ vertices in $S$.  There are at most ${|S|\choose s-r+1}$ {choices} for $A$. Then choose an $(r-1)$-element subset $B$ of $[n]$ such that $\mathcal{F}_{j}[B\cup A]$ is an $s$-clique of $\mathcal{F}_{j}$. It  can be seen that in the hypergraph $\mathcal{F}_j$, $B$ is a common neighbor of  the vertices in $A$, that is, $B\in N_{\mathcal{F}_{j}}(v)$ for any $v\in A$.
 If $A$ is a subset of $X_j$,  the number of choices for $B$ is at most $r(k-1){n-2\choose r-2}$. If $A$ is not a subset of $X_j$,  the number of choices for $B$ is at most $2(k-1){n-2\choose r-2}$. Thus,
\begin{align}\label{eqs36}
K_s^r(\mathcal{F}_i)\leq r(k-1){n-2\choose r-2}{k-1\choose s-r+1}+ 2(k-1){n-2\choose r-2}{r(k-1)\choose s-r+1}.
\end{align}
The inequality \eqref{eq4} yields
 \begin{align}\label{eq37}
 {r(k-1)\choose s-r+1} \leq (er)^{s-r+1}{k-1\choose s-r+1}.
 \end{align}
 It follows from \eqref{eqs36} and \eqref{eq37} that
 \begin{align}\label{eqs38}
K_s^r(\mathcal{F}_i)
\leq&\left(r(k-1)+2(er)^{s-r+1}(k-1)\right){k-1\choose s-r+1}{n-2\choose r-2}\nonumber\\[5pt]
\leq&3(er)^{s-r+1}(k-1){k-1\choose s-r+1}{n-2\choose r-2}.
\end{align}
Using the  inequalities \eqref{eq3} and \eqref{eq5}, we see that
\begin{align}\label{eqs39}
{n-2\choose r-2} &\leq \left(\frac{n-r}{n-k-r+2}\right)^{r-2} {n-k\choose r-2}\nonumber\\[5pt]
&=\left(1+\frac{k-2}{n-k-r+2}\right)^{r-2} {n-k\choose r-2}\nonumber\\[5pt]
&\leq \left(1+\frac{(r-2)^2(k-2)}{n-k-r+2}\right) {n-k\choose r-2}.
\end{align}
Combining \eqref{eqs38} and \eqref{eqs39}, we obtain that
\begin{align}\label{eqs40}
K_s^r(\mathcal{F}_i)\leq & 3(er)^{s-r+1}(k-1){k-1\choose s-r+1}\left(1+\frac{(r-2)^2(k-2)}{n-k-r+2}\right){n-k\choose r-2}
\end{align}
Since $n\geq 4k(t-r+2)(er)^{t-r+2}$, we find that
\begin{align}\label{eqs41}
\frac{(r-2)^2(k-2)}{n-k-r+2} \leq \frac{1}{3}
\end{align}
and
\begin{align}\label{eqs42}
4(er)^{s-r+1}(k-1)\cdot \frac{r-1}{n-k+1}\leq 1.
\end{align}
In view of \eqref{eqs40}, \eqref{eqs41} and \eqref{eqs42}, we conclude that
\[
K_s^r(\mathcal{F}_i)\leq {k-1\choose s-r+1}{n-k+1\choose r-1} \leq K_s^r(\mathcal{F}^{(r)}_{n,k-1,1}).
\]
This completes the proof.
\end{proof}

\noindent{\bf Acknowledgements.}
 We are grateful to the referee for valuable comments, leading to an improvement of an earlier version. The second author was supported by National Natural Science Foundation of China (No. 11701407) and Shanxi Province Science Foundation for Youths (No. 201801D221028 and No. 201801D221193).

\end{document}